\documentclass[12pt]{amsart}
\usepackage{amsmath}
\usepackage{amssymb}
\usepackage{mathtools,nccmath,textcomp}
\usepackage{tikz}
\usepackage[all]{xy}
\usepackage{color,soul}
\usepackage{longtable}
\allowdisplaybreaks[1]
\newtheorem{theorem}{Theorem}
\newtheorem{lemma}[theorem]{Lemma}
\newtheorem{proposition}[theorem]{Proposition}

\theoremstyle{definition}

\theoremstyle{remark}

\numberwithin{equation}{section}

\DeclareMathAlphabet{\matheur}{U}{eur}{m}{n}

\newcommand{\m}{\mathrm{m}}
\newcommand{\C}{\mathbb{C}}

\newcommand{\R}{\mathbb{R}}
\newcommand{\Q}{\mathbb{Q}}

\newcommand{\Z}{\mathbb{Z}}

\newcommand{\re}{\mathop{\mathrm{Re}}} 
\newcommand{\im}{\mathop{\mathrm{Im}}} 

\newcommand{\sgn}{\,\mathrm{sgn}}
\newcommand{\al}{\alpha}

\renewcommand\d{{\mathrm d}}

\newmuskip\pFqskip
\mathchardef\pFcomma=\mathcode`, 

\newcommand*\pFq[5]{%
  \begingroup
  \begingroup\lccode`~=`,
    \lowercase{\endgroup\def~}{\pFcomma\mkern\pFqskip}%
  \mathcode`,=\string"8000
  {}_{#1}F_{#2}\biggl(\genfrac..{0pt}{}{#3}{#4} \,\,\bigg| \,\, #5\biggr)%
  \endgroup
}

\renewcommand{\d}{\mathrm d}

\begin{document}
\title[Mahler measure of a nonreciprocal family of elliptic curves]{Mahler measure of a nonreciprocal family of elliptic curves}

\author{Detchat Samart}
\address{Department of Mathematics, Faculty of Science, Burapha University, Chonburi, Thailand 20131} \email{petesamart@gmail.com}



\date{\today}

\maketitle

\begin{abstract}
In this article, we study the logarithmic Mahler measure of the one-parameter family \[Q_\al=y^2+(x^2-\al x)y+x,\] denoted by $\m(Q_\al)$. The zero loci of $Q_\al$ generically define elliptic curves $E_\al$ which are $3$-isogenous to the family of Hessian elliptic curves. We are particularly interested in the case $\al\in (-1,3)$, which has not been considered in the literature due to certain subtleties. For $\alpha$ in this interval, we establish a hypergeometric formula for the (modified) Mahler measure of $Q_\al$, denoted by $\tilde{n}(\alpha).$ This formula coincides, up to a constant factor, with the known formula for $\m(Q_\al)$ with $|\al|$ sufficiently large. In addition, we verify numerically that if $\alpha^3$ is an integer, then $\tilde{n}(\alpha)$ is a rational multiple of $L'(E_\alpha,0)$. A proof of this identity for $\alpha=2$, which is corresponding to an elliptic curve of conductor $19$, is given.
\end{abstract}

\section{Introduction}\label{S:intro}
For any Laurent polynomial $P\in \C[x_1^{\pm1},\ldots,x_n^{\pm1}]\backslash\{0\}$, the (logarithmic) Mahler measure of $P$, denoted by $\m(P)$, is the average of $\log |P|$ over the $n$-torus. In other words, 
\begin{align*}
\m(P)&= \frac{1}{(2\pi i)^n}\idotsint\limits_{|x_1|=\dots=|x_n|=1}\log |P(x_1,\dots,x_n)|\frac{\d x_1}{x_1}\dotsb\frac{\d x_n}{x_n}.
\end{align*}

Consider the following two families of bivariate polynomials \begin{align*}
P_\al(x,y)&=x^3+y^3+1-\al xy,\\
Q_\al(x,y)&=y^2+(x^2-\al x)y+x,
\end{align*} with the parameter $\al\in \C$. For $\al\ne 3,$ the zero loci of $P_\al$ define a family of elliptic curves known as the {\it Hessian curves}. There is a $3$-isogeny between $P_\al(x,y)=0$ and the curve
\[E_\al:Q_\al(x,y)=0,\]
which is isomorphic to the curve in the Deuring form, defined by the zero locus of
\[R_\al(x,y)= y^2+\al xy+y-x^3.\]
Observe that 
\[(x^2y)^3P_\al\left(\frac{y}{x^2},\frac{1}{xy}\right)=Q_\al(x^3,y^3),\]
from which we have $\m(P_\al)=\m(Q_\al)$ (see \cite[Cor.~8]{Schinzel}). Similarly, the change of variables $(x,y)\mapsto (-y,xy)$ transforms the family $R_\al$ into $Q_\al$ without changing the Mahler measure. For some technical reasons which shall be addressed below, we will focus on the family $Q_\al$ only. Following notation in previous papers \cite{LR,Rogers,RZ}, we let
\[n(\al):= \m(Q_\al).\]
The Mahler measure of $Q_\alpha$ (and its allies) was first studied by Boyd in his seminal paper \cite{Boyd}. He verified numerically that for several $\al\in \Z$ with $\al \notin (-1,3),$
 \begin{equation}\label{E:Hesse}
 n(\al)\stackrel{?}=r_\al L'(E_\al,0),
 \end{equation}
where $r_\al\in \Q$ and $A\stackrel{?}=B$ means $A$ and $B$ are equal to at least $50$ decimal places. Later, Rodriguez Villegas \cite{RV} made an observation that \eqref{E:Hesse} seems to hold for all sufficiently large $|\al|$ which is a cube root of an integer. The values of $\al$ for which \eqref{E:Hesse} has been proven rigorously are given in Table~\ref{Ta:Pk}.

\begin{table}[ht]\label{Ta:Pk}
\centering \def\arraystretch{1.1}
    \begin{tabular}{ | c | c | c | c |}
    \hline
    $\al$ & Conductor of $E_\al$ & $r_\al$ & Reference(s)\\ \hline
    $-6$ & $27$ & $3$ & \cite{RV} \\ \hline
    $-3$ & $54$ & $1$ & \cite{Brunault} \\ \hline
    $-2$ & $35$ & $1$ & \cite{Brunault}  \\ \hline
    $-1$ & $14$ & $2$ & \cite{Mellit2},\cite{Brunault}  \\ \hline
    $\sqrt[3]{32}$ & $20$ & $\frac83$& \cite{RZ}  \\ \hline
    $\sqrt[3]{54}$ & $36$ & $\frac32$ & \cite{Rogers}\\ \hline
   $5$ & $14$ & $7$ & \cite{Mellit2}\\ \hline
    \end{tabular}
\caption{Proven formulas for \eqref{E:Hesse}}
\label{T2}
\end{table}
In addition to the results in this list, there are some known identities which relate $n(\alpha)$, where $\alpha$ is a cube root of an \textit{algebraic integer}, to a linear combination of $L$-values. For example, the author proved in \cite{SamartCJM} that the following identity is true:
\begin{equation}\label{E:Sa}
n\left(\sqrt[3]{6-6\sqrt[3]{2}+18\sqrt[3]{4}}\right)=\frac12\left(L'(F_{108},0)+L'(F_{36},0)-3L'(F_{27},0)\right),
\end{equation}
where $F_N$ is an elliptic curve over $\Q$ of conductor $N$. In compliance with Boyd's results, it is worth noting that \[\sqrt[3]{6-6\sqrt[3]{2}+18\sqrt[3]{4}}\approx 3.0005>3.\] We refer the interested reader to the aforementioned paper for more conjectural identities of this type.

Recall that a polynomial $P(x_1,x_2,\ldots,x_n)$ is said to be \textit{reciprocal} if there exist integers $d_1,d_2,\ldots,d_n$ such that 
\[x_1^{d_1}x_2^{d_2}\cdots x_n^{d_n}P(1/x_1,1/x_2,\ldots,1/x_n)=P(x_1,x_2,\ldots,x_n),\]
and \textit{nonreciprocal} otherwise. For a family of two-variable polynomials 
\begin{equation}\label{E:Pt}
\tilde{P}_\alpha(x,y)=A(x)y^2+(B(x)+\alpha x)y+C(x),
\end{equation} let $Z_\alpha$ be the zero locus of $\tilde{P}_\alpha(x,y)$ and let $K$ be the set of $\alpha\in \C$ for which $\tilde{P}_\alpha$ vanishes on the $2$-torus. Boyd conjectured from his experiments that, for all integer $\alpha$ in the unbounded component $G_\infty$ of $\C\backslash K$, if $\tilde{P}_\alpha$ is \textit{tempered} (see \cite{RV} for the definition), then $\m(\tilde{P}_\alpha)$ is related to an $L$-value of elliptic curve (if $Z_\alpha$ has genus one) or Dirichlet character (if $Z_\alpha$ has genus zero). If $\tilde{P}_\alpha(x,y)$ is reciprocal, then it can be shown that $K\subseteq \R$, implying $\overline{G}_\infty=\C$. Hence by continuity one could expect that identities like \eqref{E:Hesse} hold for all $\alpha\in \Z$, with some exceptions in the genus zero cases. Examples of polynomials satisfying these properties include the families $x+1/x+y+1/y+\alpha$ and $(1+x)(1+y)(x+y)-\alpha xy$, whose Mahler measures have been extensively studied over the past few decades (e.g. see \cite{Boyd,LSZ,LR,MS,Mellit2,Rogers,RZ,RV}).

The family $Q_\al$, on the other hand, is nonreciprocal, so the set $K$ of $\al\in \C$ for which $Q_\al$ vanishes on the $2$-torus has nonempty interior. In fact, as described in \cite[\S 2B]{Boyd} and \cite[\S 14]{RV}, $K$ is the region inside a hypocycloid whose vertices are the cube roots of $27$ in the complex plane and $K\cap \R=(-1,3)$. This is illustrated in Figure~\ref{F:hypocycloid} below. 
\begin{figure}[h!]
\centering
  \includegraphics[width=2.2 in]{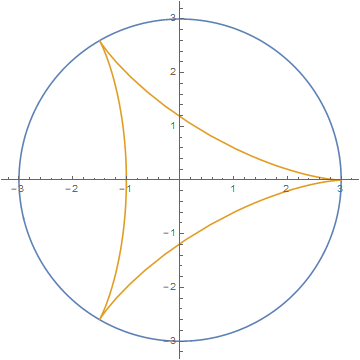}
  \caption{}
  \label{F:hypocycloid}
\end{figure} 
It is known (see, for example, \cite[Thm.~3.1]{Rogers}) that, for most complex numbers $\al$, $n(\al)$ is expressible in terms of a generalized hypergeometric function: if $|\al|$ is sufficiently large, then
\begin{equation}\label{E:naR}
n(\al)=\re\left(\log \al -\frac{2}{\al^3}\pFq{4}{3}{\frac43,\frac53,1,1}{2,2,2}{\frac{27}{\al^3}}\right).
\end{equation}
Since both sides of \eqref{E:naR} are real parts of holomorphic functions that agree at every point in an open subset of the region $\C\backslash K$, the formula \eqref{E:naR} is valid for all $\al\in \C\backslash K;$ i.e., for all $\al$ on the border and outside of the hypocycloid in Figure~\ref{F:hypocycloid}. Because of this anomalous property of the family  $Q_\al$ (and other nonreciprocal families in general), to our knowledge, there are no known results about $n(\al)$ for $\al\in K$, with an exception for the case $\alpha=0$ due to Smyth \cite{Smyth2}, namely
\[n(0)=\m(x^3+y^3+1)=\m(x+y+1)=L'(\chi_{-3},-1),\]
where $\chi_{-N}=\left(\frac{N}{\cdot}\right).$ The aim of this paper is to give a thorough investigation of these omitted values of $n(\al)$. In particular, we are interested in establishing formulas analogous to \eqref{E:Hesse} and \eqref{E:naR} for $\al \in (-1,3)$.

While the family $P_\alpha$ is more well established than the family $Q_\alpha$ in the literature, we choose to work with the latter for the following two reasons. Firstly, the family $Q_\alpha$ is in the form \eqref{E:Pt}, whose Mahler measure can be efficiently computed from both theoretical and numerical perspectives, regardless of the value of $\alpha$. Therefore, one can test the results numerically with high precision computations. The Mahler measure of $P_\alpha$, on the other hand, is quite difficult to compute, especially when $\alpha\in K$. Secondly, although the zero loci of $P_\alpha$ and $Q_\alpha$ give elliptic curves in the same isogeny class, their certain arithmetic properties, which are involved in the process of evaluating their Mahler measure in terms of $L'(E_\alpha,0)$, could be different. This will be elaborated at the end of this section.

Let us first factorize $Q_\al$ as 
\[Q_\al(x,y)=y^2+(x^2-\al x)y+x= (y-y_+(x))(y-y_-(x)),\]
where 
\[y_{\pm}(x)=-(x^2-\al x)\left(\frac12\pm \sqrt{\frac14-\frac{1}{x(x-\al)^2}}\right),\]
and denote
\[
J(\al)= \frac{1}{\pi}\int_{\cos^{-1}\left(\frac{\al-1}{2}\right)}^\pi\log |y_+(e^{i\theta})|\d\theta.
\]
(Here and throughout we use the principal branch for the complex square root.) The significance of the function $J(\al)$, which can be seen as a part of $\m(Q_\alpha)$, will be made clear later. For $\al\in (-1,1)\cup (1,3)$, $y_{\pm}(x)$ are functions on $\mathbb{T}^1:=\{x\in \C \mid |x|=1\}$. If $\al=1$, $y_{\pm}(x)$ have only one removable singularity on $\mathbb{T}^1$, namely $x=1$, so we can extend its domain to $\mathbb{T}^1$ by setting 
\[y_{\pm}(1)= \lim_{x\rightarrow 1}y_{\pm}(x)=\mp i.\]
The first main result of this paper is the following hypergeometric formula, which extends \eqref{E:naR}.

\begin{theorem}\label{T:hyper}
Let $\tilde{n}(\al)=n(\al)-3J(\al)$. For $\al\in (-1,3)\backslash\{0\}$, the following identity is true:
\[\tilde{n}(\al)=
\frac{4}{1-3\sgn(\al)}\re\left(\log \al -\frac{2}{\al^3}\pFq{4}{3}{\frac43,\frac53,1,1}{2,2,2}{\frac{27}{\al^3}}\right).
\]
\end{theorem}
\noindent By Theorem~\ref{T:hyper} and a result of Rogers \cite[Eq.~(43)]{Rogers}, we can express $\tilde{n}(\al)$ in terms of (convergent) $_3F_2$-hypergeometric series; for $\al\in (-1,3)\backslash\{0\}$, 
\[\tilde{n}(\al)=s(\alpha)\left(\frac{\sqrt[3]{2}\Gamma\left(\frac16\right)\Gamma\left(\frac13\right)\Gamma\left(\frac12\right)}{\sqrt{3}\pi^2}\al\pFq{3}{2}{\frac13,\frac13,\frac13}{\frac23,\frac43}{\frac{\al^3}{27}}+\frac{\Gamma^3\left(\frac{2}{3}\right)}{2\pi^2}\al^2\pFq{3}{2}{\frac23,\frac23,\frac23}{\frac43,\frac53}{\frac{\al^3}{27}}\right),\]
where $s(\alpha)=-\frac{(1+3\sgn(\alpha))^2}{64}.$

We also study $\tilde{n}(\al)$ from the arithmetic point of view. We discovered from our numerical computation that when $\al\in (-1,3)$ is a cube root of an integer, then $\tilde{n}(\al)$ (conjecturally) satisfies an identity analogous to \eqref{E:Hesse}. Numerical data for this identity are given in Table~\ref{T:2}. This identity can be proven rigorously in some cases using Brunault-Mellit-Zudilin's formula (see Theorem~\ref{T:BMZ} below). As a concrete example, we prove the following result.
\begin{theorem}\label{T:main2}
Let $\tilde{n}(\al)=n(\al)-3J(\al)$ and let $E_\al$ be the elliptic curve defined by the zero locus of $Q_\al$. Then the following evaluation is true:
\begin{align}
\tilde{n}(2)&= -3 L'(E_2,0).\label{E:19}
\end{align}
\end{theorem}
Note that $E_2$ has conductor $19$. What makes this curve special is that it admits {\it a modular unit parametrization}. The celebrated modularity theorem asserts that every elliptic curve over $\Q$ can be parametrized by modular functions. However, a recent result of Brunault \cite{Brunault2} reveals that there are only a finite number of them which can be parametrized by modular units (i.e. modular functions whose zeros and poles are supported at the cusps). In order to apply Brunault-Mellit-Zudilin's formula, one needs to show that the integration path corresponding to $\tilde{n}(2)$ becomes a closed path for the regulator integral defined on the curve $Q_2(x,y)=0.$ This path can then be translated into a path joining cusps on the modular curve $X_1(19)$. The calculation for this part will be worked out in Section~\ref{S:Lvalue}. On the other hand, the isogenous curve $P_2(x,y)=0$, which has Cremona label $19a1$, does not admit such a nice parametrization \cite[Tab.~1]{Brunault2}, so we cannot use the same argument to directly relate $\m(P_2)$ to $L'(E_2,0).$

\section{The hypergeometric formula}\label{S:hyper}
The goal of this section is to prove Theorem~\ref{T:hyper}. To achieve this goal, we need some auxiliary results as follows.
\begin{lemma}\label{L:ypm}
Let $\al\in \C$ and $x\in \C\backslash\{\al\}$. If $|x|=1$, then $|y_-(x)|\le 1 \le |y_+(x)|.$
\end{lemma}
\begin{proof}
Assume that $|x|=1$ and write $\sqrt{\frac14-\frac{1}{x(x-\al)^2}}=a+bi$, where $a,b\in \R.$ Since the square root is defined using the principal branch, we have $a\ge 0$. Hence 
\[|y_-(x)|=|x^2-\al x|\left|\frac12-a-bi\right|\le |x^2-\al x|\left|\frac12+a+bi\right|=|y_+(x)|.\] Since $|y_+(x)||y_-(x)|=|x|=1$, it follows that $|y_-(x)|\le 1 \le |y_+(x)|$, as desired.
\end{proof}
By Lemma~\ref{L:ypm} and Jensen's formula, we have 
\begin{equation}\label{E:na}
\begin{aligned}
n(\al)&=\frac{1}{2\pi}\int_{-\pi}^\pi \log|y_+(e^{i\theta})|\d\theta\\
&= \frac{1}{\pi}\int_{0}^\pi \log|y_+(e^{i\theta})|\d\theta\\
&= \frac{1}{\pi}\re\int_0^\pi \log\left((x-\al)\left(\frac12+\sqrt{\frac14-\frac{1}{x(x-\al)^2}}\right)\right)\bigg\rvert_{x=e^{i\theta}}\d\theta,
\end{aligned}
\end{equation}
where the second equality follows from $y_+(e^{-i\theta})=\overline{y_+(e^{i\theta})}.$ Next, we shall locate the \textit{toric points}, the points of intersection of the affine curve $Q_\al=0$ and the $2$-torus, explicitly.
\begin{proposition}\label{P:intersect}
Let $\mathbb{T}^2=\{(x,y)\in \C^2 \mid |x|=|y|=1\}$ and for each $\al\in \C$ let $C_\al=\{(x,y)\in \C^2 \mid Q_\al(x,y)=0\}$. Then for $\al\in (-1,3)$, we have
\[C_\al\cap \mathbb{T}^2= \left\{\left(e^{it},y_{\pm}(e^{it})\right)\mid t=0,\pm\cos^{-1}\left(\frac{\al-1}{2}\right)\right\}.\]  
\end{proposition}
\begin{proof}
Assume first that $\al\ne 1$. Suppose $|x|=1$, so $x=e^{it}$ for some $t\in (-\pi,\pi].$ Since $y_{\pm}(x)=-(x^2-\al x)\left(\frac12\pm \sqrt{\frac14-\frac{1}{x(x-\al)^2}}\right)$ and $|y_+(x)||y_-(x)|=|x|=1$, we have that the condition $|y_+(x)|=1=|y_-(x)|$ is equivalent to the equality
\begin{equation}\label{E:int1}
\left|\frac12+ \sqrt{\frac14-\frac{1}{x(x-\al)^2}}\right|=\left|\frac12- \sqrt{\frac14-\frac{1}{x(x-\al)^2}}\right|.
\end{equation}
It is easily seen that \eqref{E:int1} holds if and only if $\sqrt{\frac14-\frac{1}{x(x-\al)^2}}$ is purely imaginary; equivalently, $x(x-\al)^2\in (0,4)$. Simple calculation yields
\begin{align}
\re(x(x-\al)^2)&=(\cos t)((\cos t-\al)^2-\sin^2 t)-2(\cos t-\al)\sin^2 t, \label{E:int3}\\
\im(x(x-\al)^2)&=(\sin t)(2\cos t - (\al-1))(2\cos t -(\al+1)), \label{E:int4}\\
|x(x-\al)^2|&=|x-\al|^2=\al^2-2\al\cos t +1. \label{E:int2}
\end{align}
We have from \eqref{E:int4} that $x(x-\al)^2\in \R$ if and only if $\sin t=0$ or $\cos t =(\al\pm 1)/2.$ \\
If $\sin t=0$, then either $\cos t= 1$ or $\cos t= -1.$ If $\cos t=-1$, then $x(x-\al)^2=-(1+\al)^2<0.$ If $\cos t= (\al+1)/2$, then $\al\in (-1,1)$ and $\sin^2 t= 1- \left((\al+1)/2\right)^2,$ from which we can deduce using \eqref{E:int3} that \[x(x-\al)^2=\re(x(x-\al)^2)=\al-1<0.\]
Also, it can be shown using \eqref{E:int3} and \eqref{E:int2} that the remaining cases, $\cos t =1$ and $\cos t= \frac{\al-1}{2}$, imply $0< x(x-\al)^2 <4$. As a consequence, the curve $C_\al=0$ intersects $\mathbb{T}^2$ exactly at $\left(e^{it},y_{\pm}(e^{it})\right)$, where $t=0,\pm\cos^{-1}\left(\frac{\al-1}{2}\right)$. The same result also holds for $\al=1$ by continuity.
\end{proof}

\begin{lemma}\label{L:gdi}
For $\lambda\in [1,2)$, let $p_\lambda(x)=x(\lambda^2-x)\left(x^2+\left(\frac{4}{\lambda}-\lambda^2\right)x+\frac{4}{\lambda^2}\right)$ and $\gamma=\frac{\lambda^3-\lambda-2}{2\lambda}+\frac{\lambda+1}{2\lambda}\sqrt{(2-\lambda)(\lambda^3+\lambda-2)}i$. Then  we have
\begin{equation}\label{E:fml}
\int_{\lambda-1}^\gamma \frac{1}{\sqrt{-p_\lambda(x)}}\d x = \int_{0}^{-1/\lambda} \frac{1}{\sqrt{-p_\lambda(x)}}\d x, 
\end{equation}
where the left (complex) integral is path-independent in the upper-half unit disk and the right integral is a real integral.
\end{lemma}
\begin{proof}
Note first that $|\gamma|=1$ and the nonzero roots of $p_\lambda(x)$ are
\[x_1(\lambda)=\lambda^2, \quad x_2(\lambda)=\frac{\lambda^3-4+\sqrt{\lambda^3(\lambda^3-8)}}{2\lambda}, \text{ and } x_3(\lambda)=\frac{\lambda^3-4-\sqrt{\lambda^3(\lambda^3-8)}}{2\lambda},\]
which lie outside the unit circle, so the integration path for the left integral can be chosen to be any path joining $\lambda-1$ and $\gamma$ in the upper-half unit disk. For $1<\lambda<2$ and $x\in \R$,
\[x^2+\left(\frac{4}{\lambda}-\lambda^2\right)x+\frac{4}{\lambda^2}=\left(x+\left(\frac{2}{\lambda}-\frac{\lambda^2}{2}\right)\right)^2-\lambda\left(\frac{\lambda^3}{4}-2\right)>0,\]
so $-p_\lambda(x)>0$ for all $x\in (-1/\lambda,0)$ and the integral on the right-hand side is real. Define the symmetric polynomial\footnote{We obtain the polynomial $F_\lambda(x,y)$ using numerical values of the integrals in \eqref{E:fml}. The \texttt{PSLQ} algorithm plays an essential role in identifying its coefficients.} $F_\lambda(x,y)$ by
\begin{multline*}
F_\lambda(x,y):=\lambda^2(\lambda-1)x^2y^2-\lambda(\lambda-1)(\lambda^3-\lambda^2+\lambda-2)(x^2y+xy^2)+\lambda^2(x^2+y^2)\\+(\lambda^7-2\lambda^6+2\lambda^5-5\lambda^4+
6\lambda^3-6\lambda^2+6\lambda-4)xy
-2\lambda^2(\lambda-1)(x+y)+\lambda^2(\lambda-1)^2.
\end{multline*}
Then, for $\lambda\in [1,2)$, $F_\lambda(x,y)$ transforms the interval $(-1/\lambda,0)$ to a continuous path in the upper-half unit disk joining $\gamma$ and $\lambda-1$. Moreover, by implicitly differentiating $F_\lambda(x,y)=0$, it can be checked using a computer algebra system that the following equation holds on this curve:
\[\left(\frac{\d y}{\d x}\right)^2-\frac{p_\lambda(y)}{p_\lambda(x)}=0,\]
from which \eqref{E:fml} follows immediately.
\end{proof}

\begin{lemma}\label{L:dl}
For $\al\in (-1,3),$ if $\al= (\lambda^3-2)/\lambda$, then 
\begin{equation*}
\frac{\d}{\d\al}\left(n(\al)-3J(\al)\right)=-\frac{1}{\pi}\int_0^{\lambda^2} \frac{1}{\sqrt{p_\lambda(x)}}\d x,
\end{equation*}
where $p_\lambda(x)$ is defined as in Lemma~\ref{L:gdi}.
\end{lemma}
\begin{proof}
Differentiating \eqref{E:na} with respect to $\al$ yields
\begin{equation*}
\frac{\d}{\d\al}n(\al)=\frac{1}{\pi}\re\int_0^\pi \frac{\sqrt{x}}{\sqrt{x(x-\al)^2-4}}\bigg\rvert_{x=e^{i\theta}}\d\theta.
\end{equation*}
Let $c(\al)=\cos^{-1}\left(\frac{\al-1}{2}\right)$. Then, by Leibniz integral rule and Proposition~\ref{P:intersect}, we have
\begin{align*}
\frac{\d}{\d\al}J(\al)&= \frac{1}{\pi}\Bigg(-\log \left|y_+\left(e^{ic(\al)}\right)\right|\frac{\d}{\d\al}c(\al)\\
&\qquad +\re\int_{c(\al)}^\pi \frac{\d}{\d\al}\log\left((x-\al)\left(\frac12+\sqrt{\frac14-\frac{1}{x(x-\al)^2}}\right)\right)\bigg\rvert_{x=e^{i\theta}}\d\theta\Bigg)\\
&= \frac{1}{\pi}\re\int_0^{c(\al)} \frac{\sqrt{x}}{\sqrt{x(x-\al)^2-4}}\bigg\rvert_{x=e^{i\theta}}\d\theta.
\end{align*}
It follows that 
\begin{equation}\label{E:int5}
\frac{\d}{\d\al}\left(n(\al)-3 J(\al)\right)=-\frac{1}{\pi}\re\left(\left(2\int_{c(\al)}^\pi-\int_0^{c(\al)}\right)\frac{\sqrt{x}}{\sqrt{x(x-\al)^2-4}}\bigg\rvert_{x=e^{i\theta}}\d\theta\right).
\end{equation}
Let $\al=(\lambda^3-2)/\lambda$. Then $\al$ maps the interval $(1,2)$ bijectively onto $(-1,3)$ and 
\begin{equation}\label{E:poly}
x(x-\al)^2-4=(x-\lambda^2)\left(x^2+\left(\frac{4}{\lambda}-\lambda^2\right)x+\frac{4}{\lambda^2}\right).
\end{equation}
An inspection of the signs of the square roots in the integrand reveals that
\begin{align}
\int_{c(\al)}^\pi\frac{\sqrt{x}}{\sqrt{x(x-\al)^2-4}}\bigg\rvert_{x=e^{i\theta}}\d\theta &= -\int_{\gamma}^{-1}\frac{1}{\sqrt{p_\lambda(x)}}\d x=\left(\int_0^\gamma-\int_0^{-1}\right)\frac{1}{\sqrt{p_\lambda(x)}}\d x, \label{E:int6}\\
\int_{0}^{c(\al)} \frac{\sqrt{x}}{\sqrt{x(x-\al)^2-4}}\bigg\rvert_{x=e^{i\theta}}\d\theta &= \int_{1}^\gamma\frac{1}{\sqrt{p_\lambda(x)}}\d x =\left(\int_0^\gamma-\int_0^1\right)\frac{1}{\sqrt{p_\lambda(x)}}\d x, \label{E:int7}
\end{align}
where \begin{equation*}
\gamma=e^{i c(\al)}=\frac{\al-1}{2}+\frac{\sqrt{(3-\al)(\al+1)}}{2}i=\frac{\lambda^3-\lambda-2}{2\lambda}+\frac{\lambda+1}{2\lambda}\sqrt{(2-\lambda)(\lambda^3+\lambda-2)}i. 
\end{equation*}
Since $p_\lambda(x)<0$ for any $x\in(-1,0)$ and $\lambda\in (1,2)$, we have 
\begin{equation}\label{E:int8}
\re\int_0^{-1}\frac{1}{\sqrt{p_\lambda(x)}}\d x=0.
\end{equation}  
Plugging \eqref{E:int6},\eqref{E:int7}, and \eqref{E:int8} into \eqref{E:int5} gives 
\begin{equation}\label{E:int9}
\frac{\d}{\d\al}\left(n(\al)-3J(\al)\right)=-\frac{1}{\pi}\left(\int_0^1 \frac{1}{\sqrt{p_\lambda(x)}}\d x +\re\int_0^\gamma \frac{1}{\sqrt{p_\lambda(x)}}\d x\right).
\end{equation}
Note that the mapping 
\begin{equation}\label{E:MT}
x\mapsto \frac{\lambda^2-x}{\lambda x+1}
\end{equation}
is the unique M\"{o}bius transformation which interchanges the following values:
\[0\leftrightarrow \lambda^2, \quad 1 \leftrightarrow \lambda-1, \quad x_2(\lambda)\leftrightarrow x_3(\lambda),\]
where $x_2(\lambda)$ and $x_3(\lambda)$ are the roots of $x^2+(4/\lambda-\lambda^2)x+4/\lambda^2$. 
Hence using \eqref{E:MT} we have
\[\int_0^{\lambda-1}\frac{1}{\sqrt{p_\lambda(x)}}\d x=\int_1^{\lambda^2}\frac{1}{\sqrt{p_\lambda(x)}}\d x.\]
Finally, we have from Lemma~\ref{L:gdi} that
\[\int_{\lambda-1}^\gamma \frac{1}{\sqrt{p_\lambda(x)}}\d x = \int_{0}^{-1/\lambda} \frac{1}{\sqrt{p_\lambda(x)}}\d x  \in i\R, \]
so \eqref{E:int9} immediately gives the desired result.
\end{proof}

\begin{lemma}\label{L:compd}
For $\al\in (-1,0)$, we have 
\begin{equation}
\frac{\d}{\d\al}\left(n(\al)-3J(\al)\right) = \re\left(\frac{1}{\al}\pFq{2}{1}{\frac13,\frac23}{1}{\frac{27}{\al^3}}\right).\label{E:2F1m}
\end{equation}
For $\al\in (0,3)$, we have 
\begin{equation}
\frac{\d}{\d\al}\left(n(\al)-3J(\al)\right) = -2\re\left(\frac{1}{\al}\pFq{2}{1}{\frac13,\frac23}{1}{\frac{27}{\al^3}}\right).\label{E:2F1}
\end{equation}
\end{lemma}
\begin{proof}
Let us first consider \eqref{E:2F1}. We prove this identity by expressing both sides in terms of the elliptic integral of the first kind
\[K(z)=\int_0^1\frac{\d x}{\sqrt{(1-x^2)(1-z^2x^2)}}.\]
Again, let $\al= (\lambda^3-2)/\lambda.$ 
Following a procedure in \cite[Ch.~3]{Hall}, we let 
\[u=\frac{-1-\sqrt{\lambda^3+1}}{\lambda},\quad v= \frac{-1+\sqrt{\lambda^3+1}}{\lambda}, \quad x=\frac{ut-v}{t-1}.\]
This substitution transforms the integral in Lemma~\ref{L:dl} (without the factor $-1/\pi$) into
\begin{equation*}
\frac{\lambda}{2\sqrt{\lambda^3+1}}\int_{t_1}^{t_2}\frac{\d t}{\sqrt{(B_1t^2+A_1)(B_2t^2+A_2)}},
\end{equation*}
where 
\begin{align*}
t_1&=-\frac{\lambda^3+2-2\sqrt{\lambda^3+1}}{\lambda^3}, &&t_2=-t_1, \\
A_1&=\frac{\lambda^3+2-2\sqrt{\lambda^3+1}}{4\sqrt{\lambda^3+1}},&&B_1=\frac{-\lambda^3-2-2\sqrt{\lambda^3+1}}{4\sqrt{\lambda^3+1}},\\
A_2&=\frac{-\lambda^3+2+2\sqrt{\lambda^3+1}}{4\sqrt{\lambda^3+1}},&&B_2=\frac{\lambda^3-2+2\sqrt{\lambda^3+1}}{4\sqrt{\lambda^3+1}}.
\end{align*}
Observe that, for $\lambda\in(1,2)$, we have $A_1,A_2,B_2>0$, $B_1<0$, and $\sqrt{-A_1/B_1}=t_2.$ Hence the substitution $t\mapsto \sqrt{-A_1/B_1}t$ yields
\begin{align*}
\frac{\lambda}{2\sqrt{\lambda^3+1}}\int_{t_1}^{t_2}\frac{\d t}{\sqrt{(B_1t^2+A_1)(B_2t^2+A_2)}}&=\frac{\lambda}{2\sqrt{\lambda^3+1}}\sqrt{-\frac{1}{A_2B_1}}\int_{-1}^1\frac{\d t}{\sqrt{(1-t^2)\left(1-\frac{A_1B_2}{A_2B_1}t^2\right)}}\\
&=\frac{4\lambda}{\sqrt{\left(\sqrt{\lambda^3+1}+1\right)^3\left(3-\sqrt{\lambda^3+1}\right)}}K\left(\sqrt{\frac{A_1B_2}{A_2B_1}}\right).
\end{align*}
Therefore, we obtain
\begin{equation}\label{E:1}
\frac{\d}{\d\al}\left(n(\al)-3J(\al)\right)=-\frac{4\lambda}{\pi\sqrt{\left(\sqrt{\lambda^3+1}+1\right)^3\left(3-\sqrt{\lambda^3+1}\right)}}K\left(\sqrt{\frac{A_1B_2}{A_2B_1}}\right).
\end{equation}
On the other hand, we apply the hypergeometric transformation \cite[p.~410]{RZ} 
\begin{equation}\label{E:RZ}
\re \pFq{2}{1}{\frac13,\frac23}{1}{\frac{27y}{(y-2)^3}}=\frac{y-2}{y+4}\pFq{2}{1}{\frac13,\frac23}{1}{\frac{27y^2}{(y+4)^3}},
\end{equation}
which is valid for $y\in (2,8)$, to write the right-hand side of \eqref{E:2F1} as
\begin{align*}
-\frac{2}{\al}\re\left(\pFq{2}{1}{\frac13,\frac23}{1}{\frac{27}{\al^3}}\right)&=\frac{2\lambda}{2-\lambda^3}\re\left(\pFq{2}{1}{\frac13,\frac23}{1}{\frac{27\lambda^3}{(\lambda^3-2)^3}}\right)\\
&= -\frac{2\lambda}{\lambda^3+4}\pFq{2}{1}{\frac13,\frac23}{1}{\frac{27\lambda^6}{(\lambda^3+4)^3}}.
\end{align*}
The substitution $\lambda=\sqrt[3]{4(p+p^2)}$ gives a bijection from the interval $((\sqrt{3}-1)/2,1)$ onto $(\sqrt[3]{2},2)$, which is corresponding to the interval $(0,3)$ for $\al$, with the inverse mapping $p=(\sqrt{\lambda^3+1}-1)/2$. We apply this substitution together with a classical result of Ramanujan \cite[Thm~5.6]{BerndtV} to deduce
\begin{align*}
 -\frac{2\lambda}{\lambda^3+4}\pFq{2}{1}{\frac13,\frac23}{1}{\frac{27\lambda^6}{(\lambda^3+4)^3}}&=-\frac{\sqrt[3]{4(p+p^2)}}{2(p^2+p+1)}\pFq{2}{1}{\frac13,\frac23}{1}{\frac{27p^2(1+p)^2}{4(1+p+p^2)^3}}\\
 &= -\frac{\sqrt[3]{4(p+p^2)}}{2\sqrt{1+2p}}\pFq{2}{1}{\frac12,\frac12}{1}{\frac{p^3(2+p)}{1+2p}}\\
 &= -\frac{\lambda}{2\sqrt[4]{\lambda^3+1}}\pFq{2}{1}{\frac12,\frac12}{1}{\rho(\lambda)},
\end{align*}
where \[\rho(\lambda)=\frac{\lambda^6-4\lambda^3-8+8\sqrt{\lambda^3+1}}{16\sqrt{\lambda^3+1}}.\]
Then by the identities \cite[Eq.~3.2.3]{AAR}, \cite[Eq.~15.8.1]{DLMF}
\begin{equation*}
K(k)=\frac{\pi}{2}\pFq{2}{1}{\frac12,\frac12}{1}{k^2}, \qquad
K(\sqrt{r})=\frac{1}{\sqrt{1-r}}K\left(\sqrt{\frac{r}{r-1}}\right),
\end{equation*}
we arrive at
\begin{equation}\label{E:2}
-\frac{\lambda}{2\sqrt[4]{\lambda^3+1}}\pFq{2}{1}{\frac12,\frac12}{1}{\rho(\lambda)} =-\frac{4\lambda}{\pi\sqrt{\left(\sqrt{\lambda^3+1}+1\right)^3\left(3-\sqrt{\lambda^3+1}\right)}}K\left(\sqrt{\frac{\rho(\lambda)}{\rho(\lambda)-1}}\right).
\end{equation}
It can be calculated directly that 
\begin{equation*}
\frac{\rho(\lambda)}{\rho(\lambda)-1}=\frac{\lambda^6-4\lambda^3-8+8\sqrt{\lambda^3+1}}{\lambda^6-4\lambda^3-8-8\sqrt{\lambda^3+1}}=\frac{A_1B_2}{A_2B_1},
\end{equation*}
so the right-hand side of \eqref{E:2} coincides with that of \eqref{E:1} and the proof is completed.
Equation~\eqref{E:2F1m} also follows from the arguments above, provided that \eqref{E:RZ} is replaced with 
\begin{equation*}
\re \pFq{2}{1}{\frac13,\frac23}{1}{\frac{27y}{(y-2)^3}}=\frac{4-2y}{y+4}\pFq{2}{1}{\frac13,\frac23}{1}{\frac{27y^2}{(y+4)^3}},
\end{equation*}
which is valid for $y\in (1,2).$
\end{proof}

\begin{proof}[Proof of Theorem~\ref{T:hyper}]
For $\al>3$, we can apply term-by-term differentiation to show that 
\[\frac{\d}{\d\al}\re\left(\log \al -\frac{2}{\al^3}\pFq{4}{3}{\frac43,\frac53,1,1}{2,2,2}{\frac{27}{\al^3}}\right)=\re\left(\frac{1}{\al}\pFq{2}{1}{\frac13,\frac23}{1}{\frac{27}{\al^3}}\right).\]
By analytic continuation, the above equality also holds for $\al\in (-1,0)\cup (0,3)$. Therefore, integrating both sides of \eqref{E:2F1m} and \eqref{E:2F1} yields
\[n(\al)-3J(\al)= \begin{cases}
\re\left(\log \al -\frac{2}{\al^3}\pFq{4}{3}{\frac43,\frac53,1,1}{2,2,2}{\frac{27}{\al^3}}\right)+C_1, & \text{ if } -1<\al<0,\\
-2\re\left(\log \al -\frac{2}{\al^3}\pFq{4}{3}{\frac43,\frac53,1,1}{2,2,2}{\frac{27}{\al^3}}\right)+C_2, & \text{ if } 0<\al<3,
\end{cases}
\]
for some constants $C_1$ and $C_2$. Since $\al=-1$ and $\al=3$ are on the boundary of the set $K$ defined in Section~\ref{S:intro}, an argument underneath \eqref{E:naR} implies that
\begin{equation}\label{E:n3}
\begin{aligned}
n(-1)&= \re\left(\log(-1)+2\pFq{4}{3}{\frac43,\frac53,1,1}{2,2,2}{-27}\right),\\
n(3)&= \re\left(\log 3 -\frac{2}{27}\pFq{4}{3}{\frac43,\frac53,1,1}{2,2,2}{1}\right).
\end{aligned}
\end{equation}
Hence, by continuity of $n(\al)$ and \eqref{E:n3}, we have
\begin{align*}
C_1&=\lim_{\al\rightarrow -1^+}(-3J(\al))=0, \\
C_2&=3\lim_{\al\rightarrow 3^-}(n(3)-J(\al))=0,
\end{align*}
and the desired result follows.
\end{proof}

\section{Relation to elliptic regulators and $L$-values} \label{S:Lvalue}
In this section, we prove Theorem~\ref{T:main2}, which resembles Boyd's conjectures \eqref{E:Hesse}. The key idea of the proof is to rewrite $\tilde{n}(\al)$ as a regulator integral over a path joining two cusps and apply Brunault-Mellit-Zudilin formula \cite{Zudilin}, which is stated below. As usual, we define the real differential
form $\eta(f,g)$ for meromorphic functions $f$ and $g$ on a smooth curve $C$ as
\[\eta(f,g)=\log |f|\d\arg(g)-\log|g|\d\arg(f),\]
where $\d \arg(g)=\im(\d g/g)$. 
\begin{theorem}[Brunault-Mellit-Zudilin]\label{T:BMZ}
Let $N$ be a positive integer and define
\begin{equation*}
g_a(\tau)= q^{NB_2(a/N)/2}\prod_{\substack{n \geq 1 \\ n \equiv a \bmod N}}(1-q^n)\prod_{\substack{n \geq 1 \\ n \equiv -a \bmod N}}(1-q^n), \qquad q:= e^{2\pi i \tau},
\end{equation*}
where $B_2(x)= \{x\}^2-\{x\}+1/6$. Then for any $a,b,c\in \Z$ such that $N \nmid ac$ and $N\nmid bc$,
\[\int_{c/N}^{i\infty} \eta(g_a,g_b)=\frac{1}{4\pi}L(f(\tau)-f(i\infty),2),\]
where $f(\tau)=f_{a,b;c}(\tau)$ is a weight $2$ modular form given by 
\[f_{a,b;c}=e_{a,bc}e_{b,-ac}-e_{a,-bc}e_{b,ac}\]
and 
\[e_{a,b}(\tau)=\frac{1}{2}\left(\frac{1+\zeta_N^a}{1-\zeta_N^a}+\frac{1+\zeta_N^b}{1-\zeta_N^b}\right)+\sum_{m,n\ge 1}\left(\zeta_N^{am+bn}-\zeta_N^{-(am+bn)}\right)q^{mn}, \quad \zeta_N:= e^{\frac{2\pi i}{N}}.\]
\end{theorem}
Let us first outline a general framework for computing $\tilde{n}(\al)$ in terms of a regulator integral. Recall from Deninger's result \cite[Prop.~3.3]{Deninger} that if $Q_\al(x,y)$ is irreducible, then 
\begin{equation*}
n(\al)=-\frac{1}{2\pi}\int_{\overline{\gamma}_\al}\eta(x,y),
\end{equation*}
where $\gamma_\al$ is the Deninger path on the curve $E_\al: Q_\al(x,y)=0$; i.e., \[\gamma_\al=\{(x,y)\in \C^2 \mid |x|=1, |y|>1, Q_\al(x,y)=0\}.\]
If $Q_\al$ does not vanish on the torus, then $\overline{\gamma}_\al$ becomes a closed path, so  the Bloch-Beilinson conjectures give a prediction that \eqref{E:Hesse} holds for all sufficiently large $|\al|$ with suitable arithmetic properties; in this case, we need that $\al$ be a cube root of an integer. On the other hand, if $\al\in (-1,3)$, then the functions $y_{\pm}(x)$ defined in Section~\ref{S:intro} are discontinuous at the toric points as given in Proposition~\ref{P:intersect}, so $\overline{\gamma}_\al$ is not closed in this case. We will show, however, that the path on $E_\al$ corresponding to $\tilde{n}(\al)$ is indeed closed, so that $\tilde{n}(\al)$ is (conjecturally) related to $L$-values. The numerical data supporting this hypothesis are given in Table~\ref{T:2}.
\begin{lemma}\label{L:path}
Let $\alpha\in (-1,3)$ and let $\tilde{n}(\alpha)=n(\alpha)-3J(\alpha).$ Then 
\[\tilde{n}(\alpha)=-\frac{1}{2\pi}\int_{\tilde\gamma_\alpha}\eta(x,y)\]
for some $\tilde\gamma_\alpha\in H_1(E_\alpha,\Z)^-.$ In other words, the integration path associated to the \textit{modified} Mahler measure $\tilde{n}(\alpha)$ can be realized as a closed path which is anti-invariant under complex conjugation.
\end{lemma}
\begin{proof}
We label the six toric points obtained from Proposition~\ref{P:intersect} as follows: 
\begin{align*}
P_1^{\pm}&=(1,y_{\pm}(1))=\left(1, Y_{\pm}\right),\\
P_2^{\pm}&=(e^{\pm ic(\alpha)},y_+(e^{\pm ic(\alpha)}))= \left(Y_{\pm},1\right),\\
P_3^{\pm}&=(e^{\pm ic(\alpha)},y_-(e^{\pm ic(\alpha)}))=\left(Y_{\pm}, Y_{\pm}\right),
\end{align*}
where $c(\alpha)=\cos^{-1}\left(\frac{\alpha-1}{2}\right)$ and \[Y_{\pm}=\frac{\al-1}{2}\pm\frac{\sqrt{(3-\al)(\al+1)}}{2}i.\]
Observe that $\tilde{n}(\alpha)$ can be rewritten as $\tilde{n}(\alpha)=I(\alpha)-2J(\alpha),$
where 
\begin{align*}
I(\alpha)&=\frac{1}{2\pi}\int_{-c(\alpha)}^{c(\alpha)}\log|y_+(e^{i\theta})|\d\theta,\\
J(\alpha)&=\frac{1}{2\pi}\int_{c(\alpha)}^{2\pi-c(\alpha)}\log|y_+(e^{i\theta})|\d\theta.
\end{align*}
Let $S=\{P_1^{\pm},P_2^{\pm},P_3^{\pm}\}$. Then we may identify the paths corresponding to $I(\alpha)$ and $J(\alpha)$ as elements in the relative homology $H_1(E_\alpha,S,\Z)$, say $\gamma_I$ and $\gamma_J$, respectively. In other words, we write
\[I(\alpha)=-\frac{1}{2\pi}\int_{\gamma_I}\eta(x,y),\quad J(\alpha)=-\frac{1}{2\pi}\int_{\gamma_J}\eta(x,y),\]
and boundaries of these paths can be seen as $0$-cycles on $S$. Computing the limits of $y_{+}(e^{i\theta})$ as $\theta$ approaches $0,c(\alpha)$, and $-c(\alpha)$ from both sides, we find that
\begin{align*}
\lim_{\theta\rightarrow -c(\alpha)^+}y_+(e^{i\theta})&=\lim_{\theta\rightarrow c(\alpha)^-}y_+(e^{i\theta})=1,\\
\lim_{\theta\rightarrow 0^+}y_+(e^{i\theta})&=Y_-,\\
\lim_{\theta\rightarrow 0^-}y_+(e^{i\theta})&=Y_+.
\end{align*}
Therefore, the path $\gamma_I$ is discontinuous at $\theta=0$ and 
\begin{equation}\label{E:bgI}
\partial \gamma_I=[[P_1^+]-[P_2^-]]+[[P_2^+]-[P_1^-]].
\end{equation}
This is illustrated in Figure~\ref{F:yp} for $\alpha=2$, where the dashed curves in the upper-half plane and the lower-half plane, both oriented counterclockwise, correspond to $\theta\in (-c(\alpha),0)$ and $\theta\in (0,c(\alpha))$, respectively. 
\begin{figure}[!htb]
   \begin{minipage}{0.48\textwidth}
     \centering
  \includegraphics[width=.7\linewidth]{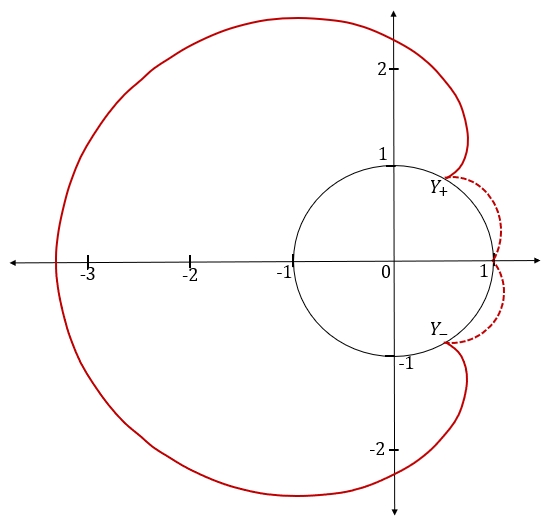}
  \caption{$y_+(e^{i\theta}), \,\theta\in [0,2\pi)$}
  \label{F:yp}
  \end{minipage}\hfill
   \begin{minipage}{0.48\textwidth}
     \centering
  \includegraphics[width=.7\linewidth]{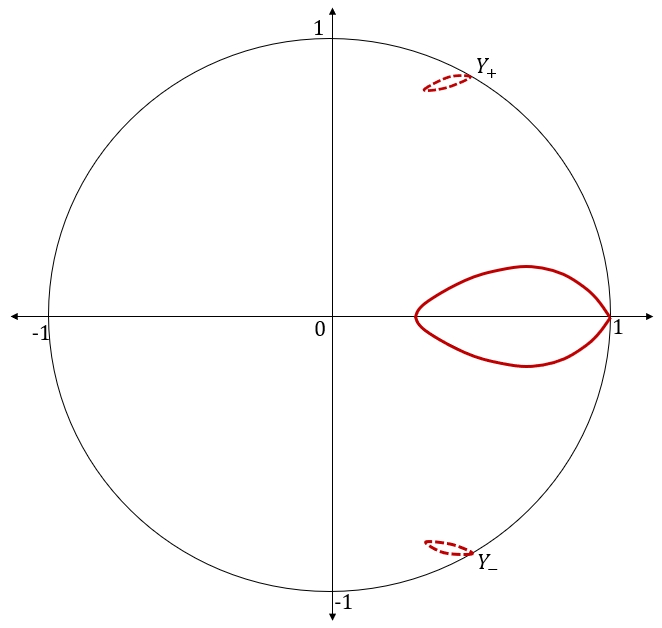}
  \caption{$y_-(e^{i\theta}), \,\theta\in [0,2\pi)$}
  \label{F:ym}
   \end{minipage}
\end{figure} 
Next, observe that 
\begin{equation*}
\lim_{\theta\rightarrow c(\alpha)^+}y_-(e^{i\theta})=1=\lim_{\theta\rightarrow -c(\alpha)^-}y_-(e^{i\theta}),
\end{equation*}
and $\gamma_J$ can be identified as the path $\{(e^{i\theta},y_-(e^{i\theta}))\mid c(\alpha)<\theta<2\pi - c(\alpha)\}$ (with reversed orientation), implying 
\begin{equation}\label{E:bgJ}
\partial \gamma_J=[[P_2^+]-[P_2^-]].
\end{equation}
(For $\alpha=2$, the $y$-coordinate of this path is the bold curve inside the unit circle, as illustrated in Figure~\ref{F:ym}, oriented clockwise.) Define 
\[\gamma_J'=\left\{\left(\frac{1}{y_-\left(e^{i\theta}\right)},y_-\left(\frac{1}{y_-\left(e^{i\theta}\right)}\right)\right) \mid c(\alpha)<\theta <2\pi-c(\alpha)\right\}.\]
By some calculation, one sees that
\begin{align*}
y_-\left(\frac{1}{y_-\left(e^{i\theta}\right)}\right)&=e^{-i\theta}, \\
\lim_{\theta\rightarrow c(\alpha)^+}y_-(e^{i\theta})&=1=\lim_{\theta\rightarrow -c(\alpha)^-}y_-(e^{i\theta}),
\end{align*}
implying 
\begin{equation}\label{E:bgJp}
\partial \gamma_J' = [[P_1^+]-[P_1^-]].
\end{equation}
Moreover, we have 
\begin{align*}
\int_{\gamma_J}\eta(x,y)&=-\int_{c(\alpha)}^{2\pi-c(\alpha)}\log|y_+\left(e^{i\theta}\right)|\d\theta\\
&=\int_{c(\alpha)}^{2\pi-c(\alpha)}\log|y_-\left(e^{i\theta}\right)|\d\theta\\
&=\int_{c(\alpha)}^{2\pi-c(\alpha)}\log(|1/y_-\left(e^{i\theta}\right)|)\d(-\theta)\\
&=\int_{\gamma_J'}\eta(x,y).
\end{align*}
Finally, we arrive at 
\[\tilde{n}(\alpha)=I(\alpha)-2J(\alpha)=-\frac{1}{2\pi}\left(\int_{\gamma_I}\eta(x,y)-\int_{\gamma_J}\eta(x,y)-\int_{\gamma_J'}\eta(x,y)\right)=-\frac{1}{2\pi}\int_{\tilde{\gamma}_\alpha}\eta(x,y),\]
where, by \eqref{E:bgI},\eqref{E:bgJ}, and \eqref{E:bgJp}, $\tilde{\gamma}_\alpha$ has trivial boundary, from which we can conclude that $\tilde{\gamma}_\alpha\in H_1(E_\alpha,\Z).$ It is clear from the construction of the paths $\gamma_I,\gamma_J,$ and $\gamma_J'$ that they are anti-invariant under the action of complex conjugation. Therefore, we have $\tilde{\gamma}_\alpha\in H_1(E_\alpha,\Z)^-$, as desired.
\end{proof}

We shall use Theorem~\ref{T:BMZ} and Lemma~\ref{L:path} to prove Theorem~\ref{T:main2}. We essentially follow an approach of Brunault \cite{Brunault} in identifying the path $\tilde{\gamma}_2$ as the push-forward of a path joining cusps on $X_1(19)$ with the aid of \texttt{Magma} and \texttt{Pari/GP}.

\begin{proof}[Proof of Theorem~\ref{T:main2}]
The elliptic curve $E_2:y^2+(x^2-2x)y+x=0$ has Cremona label $19a3$, so it admits a modular parametrization $\varphi:X_1(19)\rightarrow E_2$. Let $f_{2}$ be the weight $2$ newform of level $19$ associated to the curve $E_2$ and let $\omega=2\pi i f_2(\tau)\d\tau,$ the pull-back of the holomorphic differential form on $E_2$. Using Magma and Pari/GP codes in \cite[\S 6.1]{Brunault}, we find that \[\int_{4/19}^{-4/19}\omega = -\Omega^-\approx - 4.12709i,\] where $\Omega^-$ is the imaginary period of $E_2$ obtained by subtracting twice the complex period from the real period of $E_2$. Hence it follows that $\tilde{\gamma}_2= \varphi_*\left\{\frac{4}{19},-\frac{4}{19}\right\}$, where $\tilde{\gamma}_\alpha$ is the path associated to $\tilde{n}(\alpha)$. We have from \cite[Tab.~1]{Brunault2} that the curve $E_2$ can be parametrized by modular units, which are given explicitly as follows.
Let
\begin{align*}
x(\tau)&=-\frac{g_1g_7g_8}{g_2g_3g_5},\\
y(\tau)&=\frac{g_1g_7g_8}{g_4g_6g_9},
\end{align*}
where $g_a:=g_a(\tau)$ is as given in Theorem~\ref{T:BMZ} with $N=19$. By a result of Yang \cite[Cor.~3]{Yang}, both $x(\tau)$ and $y(\tau)$ are modular functions on $\Gamma_1(19)$. Multiplying each term by a modular form in $M_2(\Gamma_1(19))$, one can apply Sturm's theorem \cite[Cor.~9.19]{Stein}, with the Sturm bound $B(M_2(\Gamma_1(19)))=60$, to show that $y(\tau)^2+(x(\tau)^2-2x(\tau))y(\tau)+x(\tau)$ vanishes identically; i.e., $(x(\tau),y(\tau))$ parametrizes the curve $E_2$.
Finally, by Lemma~\ref{L:path} and Theorem~\ref{T:BMZ}, we find that 
\begin{equation*}
\tilde{n}(2)=-\frac{1}{2\pi}\int_{\tilde{\gamma}_2}\eta(x,y)=\frac{1}{2\pi}\int_{-4/19}^{4/19}\eta(x(\tau),y(\tau))=-\frac{1}{4\pi^2}L(57f_{2},2)=-3L'(f_2,0),
\end{equation*}
where the last equality follows from the functional equation for $L(f_2,s)$.
\end{proof}

In addition to \eqref{E:19}, we discovered that, for all $\al\in (-1,3)$ which are cube roots of integers, the following identity holds numerically:
\begin{equation}\label{E:conj}
\tilde{n}(\al)\stackrel{?}=r_\al L'(E_\al,0),
\end{equation}
where $r_\al\in \Q$. The data of $r_\al$ and $E_\al$ are given in Table~\ref{T:2}.

\begin{table}[ht]\label{Ta:ntilde}
\centering \def\arraystretch{1.1}
    \begin{tabular}{ | c | c | c | c | c | c |}
    \hline
    $\al^3$ & Cremona label of $E_\al$ & $r_\al$ & $\al^3$ & Cremona label of $E_\al$ & $r_\al$ \\ \hline
    $1$ & $26a3$ & $-1$  & $14$ & $2548d1$ & $1/36$ \\ \hline
    $2$ & $20a1$ & $-5/3$ & $15$ & $1350i1$ & $1/18$ \\ \hline
    $3$ & $54a1$ & $-2/3$ & $16$ & $44a1$ & $-4/3$ \\ \hline
    $4$ & $92a1$ & $-1/3$  & $17$ & $2890e1$ & $-1/27$\\ \hline
    $5$ & $550d1$ & $-1/9$  & $18$ & $324b1$ & $-1/6$\\ \hline
    $6$ & $756f1$ & $-1/18$  & $19$ & $722a1$ & $1/9$\\ \hline
    $7$ & $490a1$ & $1/9$ & $20$ & $700i1$ & $-1/9$ \\ \hline
    $8$ & $19a3$ & $-3$  & $21$ & $2464k1$ & $-1/27$\\ \hline
    $9$ & $162c1$ & $-1/3$  & $22$ & $2420d1$ & $1/26$\\ \hline
    $10$ & $1700c1$ & $1/36$  & $23$ & $1058b1$ & $-1/12$\\ \hline
    $11$ & $242b1$ & $-1/3$ & $24$ & $27a1$ & $-3$ \\ \hline
    $12$ & $540d1$ & $1/9$ & $25$ & $50a1$ & $-5/3$ \\ \hline
    $13$ & $2366d1$ & $-1/45$ & $26$ & $676c1$ & $-1/6$ \\ \hline
    \end{tabular}
\caption{Data for \eqref{E:conj}}\label{T:2}
\end{table} 
It might be possible to prove some formulas in this list by relating $\tilde{n}(\al)$ to known results in Table~\ref{Ta:Pk}. In particular, the conjectural formulas for the curves of conductor $20, 27,$ and $54$ are equivalent to the following identities:
\begin{align*}
\tilde{n}(\sqrt[3]{2})&\stackrel{?}=-\frac58 n(\sqrt[3]{32}), \\
\tilde{n}(\sqrt[3]{24})&\stackrel{?}=- n(-6),\\
\tilde{n}(\sqrt[3]{3})&\stackrel{?}=-\frac32 n(-3).
\end{align*}
As a side note, the authors of \cite{RZ} (incorrectly) proved 
\begin{equation}\label{E:RZ2}
n(\sqrt[3]{2})=\frac{5}{6}L'(E_{\sqrt[3]{2}},0)
\end{equation}
(see the corollary under \cite[Thm.~5]{RZ}).
In their arguments, they made use of the following functional identity for Mahler measures \cite[Thm.~2.4]{LR}: for sufficiently small $|p|\ne 0$,
\begin{equation}\label{E:FE}
3g\left(\frac{1}{p}\right)=n\left(\frac{1+4p}{\sqrt[3]{p}}\right)+4n\left(\frac{1-2p}{\sqrt[3]{p^2}}\right),
\end{equation}
where $g(\alpha)=\m((x+1)(y+1)(x+y)-\alpha xy)$. When any of the arguments of $n$ in \eqref{E:FE} enters the region inside the hypocycloid in Figure~\ref{F:hypocycloid} (e.g. $p=-1/2$ in this case), this functional identity could be invalid due to discontinuity. Therefore, it is logically forbidden to deduce \eqref{E:RZ2} from \eqref{E:FE}. In fact, by extending the hypergeometric formula \eqref{E:naR} to the real line, Rogers \cite{RogersP} conjectured that 
\begin{equation}\label{E:R2}
n(2)=\frac{3}{2}L'(E_2,0),
\end{equation}  
which is not the case by Theorem~\ref{T:main2}. It should be noted that both \eqref{E:RZ2} and \eqref{E:R2} make perfect sense if one thinks of $n(\alpha)$ as the right-hand side of $\eqref{E:naR}$ on the punctured real line. That said, this strange behavior of the function $n(\al)$ became a part of our motivation to initiate this project.

\section{Final remarks} \label{S:Remarks}
The family $Q_\alpha$ is among the several nonreciprocal families of two-variable polynomials studied by Boyd. Our results provide evidence of how Mahler measure behaves when the zero locus of a bivariate polynomial intersects the $2$-torus nontrivially. This could shed some light on the discrepancies between Mahler measure and (elliptic) regulator, which is conjecturally related to $L$-values under favorable conditions. Another family which possesses similar properties (i.e. nonreciprocality and temperedness) to $Q_\alpha$ is 
\[S_\alpha = y^2+(x^2+\alpha x +1)y+x^3,\]
which is labeled $(2\text{-}33)$ in \cite{Boyd}. Let $K$ be as defined in Section~\ref{S:intro}. Then for the family $S_\alpha$ we have $K\cap \R=[-4,2].$ For $\alpha$ in this range, the Mahler measure of $S_\alpha$ again splits naturally at the points of intersection between the curve $S_\alpha=0$ and the $2$-torus. If $k=0$, these points are $\pm i$, and Boyd verified numerically that 
\begin{equation}\label{E:B11}
\frac{1}{\pi}\int_{0}^{\pi/2}\log |y_-(e^{i\theta})|\d \theta-\frac{1}{\pi}\int_{\pi/2}^{\pi}\log |y_-(e^{i\theta})|\d \theta\stackrel{?}=-L'(E,0),
\end{equation}
where $y_-(x)=-\frac{(x^2+1)}{2}\left(1-\sqrt{1-\frac{4x^3}{(x^2+1)^2}}\right)$ and $E$ is the conductor $11$ elliptic curve defined by $S_0=0$. He also remarked
\begin{center}
\textit{
``This is in accord with our contention that in case $P$ vanishes on the torus, it is the integral of $\omega$ around a branch cut rather than $\m(P),$ which should be rationally
related to $L'(E,0)$.''.}
\end{center}
One might try to prove this identity using the investigation carried out in Section~\ref{S:Lvalue} and a result of Brunault \cite{Brunault11} concerning Mahler measure of a conductor $11$ elliptic curve. We also discovered conjectural identities analogous to \eqref{E:B11} for elliptic curves of conductor $17$ and $53$, which are corresponding to $k=1$ and $k=-1$, respectively. As opposed to the family $Q_\alpha$, we are unable to find a general formula, both analytically and arithmetically, for Mahler measure (or its modification) of $S_\alpha$, so the situation seems less apparent for this family. 

We would also like to point out another related result in the literature which we find incomplete. In \cite[Thm~3.1]{GR}, Guillera and Rogers assert that for $q=e^{2\pi i\tau}\in (-1,1)$ if $\alpha=3\left(1+27\frac{\eta^{12}(3\tau)}{\eta^{12}(\tau)}\right)^{\frac13},$ then
\begin{equation}\label{E:ED}
{n(\alpha)= \frac{9}{2\pi}\sum_{n=-\infty}^\infty D\left(e^{2\pi i/3}q^n\right)},
\end{equation}
where $\eta(\tau)$ is the Dedekind eta function, and $D(z)$ is the Bloch-Wigner dilogarithm. The summation in the formula above can be seen as a value of the \textit{elliptic dilogarithm}. Consider the curve $E_2$, which appears in Theorem~\ref{T:main2} and is isomorphic to $\C/\Z+\Z\tau$, where $\tau=1/2+0.50586\ldots i.$ Then we have $q=e^{2\pi i \tau}=-0.04165\ldots$. However, the identity \eqref{E:ED} seems invalid in this case (and all other cases for $-1< \alpha< 3)$. The right-hand side is numerically equal to $\frac{3}{2}L'(E_2,0)$, which is a conjecture of Bloch and Grayson \cite{BG}, while $n(2)$ is not a rational multiple of $L'(E_2,0)$. A correct formula for $\alpha\in (-1,3)$ should be
\begin{equation*}
\tilde{n}(\alpha)=
-\frac{9}{\pi}\displaystyle\sum_{n=-\infty}^\infty D\left(e^{2\pi i/3}q^n\right),
\end{equation*}
which can be proven using Lemma~\ref{L:path} and \cite[Prop.~19]{Brunault}.

Finally, we propose some problems for the interested readers. 
\begin{itemize}
\item[(i)] The function $\tilde{n}(\alpha)$ looks somewhat unnatural at first glance. Is it possible to write it as the (full) Mahler measure of some polynomial?
\item[(ii)] Do there exist algebraic integers $\beta$ for which $\sqrt[3]{\beta}\in (-1,3)$ and $\tilde{n}(\sqrt[3]{\beta})$ is a linear combination of $L'(E,0)$ (i.e. identities analogous to \eqref{E:Sa})? As suggested by a result of Guillera and Rogers above, one might start by evaluating the function $u(\tau)=3\left(1+27\frac{\eta^{12}(3\tau)}{\eta^{12}(\tau)}\right)^\frac13$ at some suitable CM points and numerically compare $\tilde{n}(u(\tau))$ with related elliptic $L$-values using the \texttt{PSLQ} algorithm.
\end{itemize}
\section*{Funding}
This work was supported by the National Research Council of Thailand (NRCT) under the Research Grant for Mid-Career Scholar [N41A640153 to D.S.].

\section*{Acknowledgements}
The author is indebted to Wadim Zudilin for helpful discussions and his suggestion about integral and hypergeometric identities in the proofs of Lemma~\ref{L:dl} and Lemma~\ref{L:compd}. The author would also like to thank Fran\c{c}ois Brunault for his guidance on an approach to proving Lemma~\ref{L:path} and his explanation about Deninger's results. This work would not have been complete without insightful comments from Mat Rogers and Fran\c{c}ois Brunault on early versions of this manuscript, so the author would like to acknowledge them here. Finally, the author thanks Yusuke Nemoto and Zhengyu Tao for bringing a sign error in Theorem~\ref{T:hyper}  and a miscalculation in the proof of Lemma~\ref{L:path} in the previous version of this paper to his attention.
\bibliographystyle{amsplain}
\bibliography{Mahler}
\end{document}